\documentclass[preprint]{imsart}

\usepackage{amssymb}
\usepackage{amsmath}
\usepackage{amsthm}

\startlocaldefs
\numberwithin{equation}{section}
\theoremstyle{plain}
\newtheorem{lem}{Lemma}[section]
\newtheorem{df}{Definition}[section]
\newtheorem{tw}{Theorem}[section]

\endlocaldefs

\begin{document}

\begin{frontmatter}

\title{Almost sure asymptotic properties of central order statistics from stationary processes}

\runtitle{Almost sure asymptotic properties of central order statistics}

\begin{aug}
 \author{\fnms{Aneta} \snm{Augustynowicz}\ead[label=e1]{a.augustynowicz@mini.pw.edu.pl}}
%\and
%\author{\fnms{Anna} \snm{Dembi\'{n}ska}\corref{}\ead[label=e2]{dembinsk@mini.pw.edu.pl}}

\thankstext{t2}{%First supporter of the project
The work was supported by Warsaw University of Technology under grant no. 504/03934/1120.
}
%\thankstext{t3}{Second supporter of the project}
\runauthor{A. Augustynowicz}
%\runauthor{F. Author et al.}

\affiliation{Warsaw University of Technology}

\address{
Faculty of Mathematics and Information Science\\
Warsaw University of Technology\\
ul. Koszykowa 75, 00-662 Warsaw, Poland\\
\printead{e1}
%\phantom{E-mail:\ }\printead*{e1}
}
\end{aug}

\begin{abstract}
In this paper, we formulate and prove new properties of conditional quantiles given one of the particular sigma-fields. Next, we use them to investigate almost sure asymptotic behavior of central order statistics which arise from strictly stationary processes. Specifically we provide a~new version of a~strong ergodic theorem for central order statistics.
\end{abstract}

%\begin{keyword}[class=MSC]
%\kwd[Primary ]{}
%\kwd{}
%\kwd[; secondary ]{}
%\end{keyword}


\end{frontmatter}

\section{Introduction}
\label{sec1}

Let $(X_n,n\ge 1)$ be a~sequence of random variables (rv's) with
common cumulative distribution function (cdf) $F$ and $X_{1:n}\le\ldots\le X_{n:n}$ be the order statistics corresponding to the sample $(X_1,\ldots,X_n)$. Our purpose is to investigate asymptotic properties of $(X_{k_n:n},n\ge 1)$ in the case when
\begin{equation}\label{warK}
\textrm{for all }n,\ 1\le k_n\le n\textrm{ and }\lim_{n\to\infty} k_n/n=\lambda\in(0,1).
\end{equation}
Then $(X_{k_n},n\ge 1)$ is called a~sequence of central order statistics. If $\lim_{n\to\infty} k_n/n$ $=\lambda\in\{0,1\}$ then $X_{k_n:n},n\ge 1$ are referred to extreme or intermediate order statistics, but we will restrict our attention to the former case. Analogous results for the latter case can be found in Buraczy\'nska and Dembi\'nska (2018) and Dembi\'nska and Buraczy\'nska (2019).

Asymptotic properties of $(X_{k_n:n},n\ge 1)$, when~\eqref{warK} holds, have been under scientific considerations over the last few decades. One of the oldest results was given by Bahadur (1966) who provided an asymptotic almost sure representation for sample quantiles of independent and identically distributed (iid) rv's via the empirical distribution function. However, this result was based on specific restrictions imposed on the cdf~$F$. Bahadur's representation has been generalized by several authors, for instance by Ghosh~(1971), Kiefer~(1976) and Wu~(2005).

Another well-known result was derived by Smirnov~(1952) who proved that if~\eqref{warK} is satisfied then 
$X_{k_n:n}$ is a~strongly consistent estimator of the $\lambda$th population quantile provided that this quantile is unique and the rv's $X_n,n\ge 1$ are iid. 
Dembi\'nska~(2014) extended this result 
to the case when $(X_n,n\ge 1)$ is an arbitrary strictly stationary sequence. Specifically she gave sufficient conditions for the almost sure convergence of $X_{k_n:n}$ as $n\to\infty$ and described the distribution of the limiting rv.
For this purpose she used a~concept defined by Tomkins (1975) - the conditional quantile of an rv given a~sigma-field.
Even though her theorem concerns the asymptotic behavior of central order statistics from the sequence $\mathbb{X}=(X_{n},n\ge 1)$, its assumption and conclusion use an~auxiliary rv $Y$ from the probability triple $(\mathbb{R}^{\mathbb{N}},\mathcal{B}(\mathbb{R}^{\mathbb{N}}),\mathbb{Q})$. 
The aim of the present paper is to provide an elegant version
of this theorem, which is expressed only in terms of~$\mathbb{X}$.
Before we state and prove this version, we analyze conditional quantiles and obtain several properties of this concept with respect to one of the particular sigma-fields.

Conditional quantile, which is one of the most significant tool in our analysis, has numerous applications in statistics and finance. Its estimation allows solving such problems as the relationship between expected value and volatility (Glosten et al. 1993) or determining the CoVar systematic risk measure (Engle and Manganelli 2004). 
It was also applied to describe almost sure asymptotic behavior of proportions of near-order-statistic observations in sample from stationary processes (Dembi\'nska 2017).
Furthermore, statisticians use conditional quantiles for robust beta estimation (Chan and Lakonishok 1992) and economists to estimate various phenomenons, in particular, to measure macroecomics risk (Boucher and Maillet 2013), wage structure (Buchinsky and Leslie 2010) and economic growth (Castellano and Ho 2013). Therefore, methods for estimating conditional quantiles have been well-developed. The most popular are quantile regression (Koenker
2005, Koenker and Bassett 1978), local quantile regression (Spokoiny et al. 2013) and non-parametric estimation of conditional quantiles (Li and Racine 2007).

This paper is organized as follows. In Section~2, we recall some facts from the ergodic theory and the concept of conditional quantile. Then, in Section~3, we formulate and prove the exposition of new properties of conditional quantiles. 
%These facts are relevant to ensure the more elegant version of theorem describing asymptotic properties of sequence $(X_{k_n:n},n\ge 1)$ given by Dembi\'nska (2014). Formulation of this theorem is presented in Section~4. 
In Section~4, we use these properties to obtain the main result of this paper, i.e. to derive a~refinement of the strong ergodic theorem for central order statistics given by Dembi\'nska (2014).

Finally, we introduce our notation. Unless otherwise stated, the rv's $X_n,n\ge 1$, exist in a~probability space $(\Omega,\mathcal{F},\mathbb{P})$. $\mathbb{R}$ and $\mathbb{N}$ correspond to the sets of real numbers and positive integers, respectively. We use the symbol $I(\cdot)$ to denote the indicator function, that is $I(x\in A)=1$ if $x\in A$ and $I(x\in A)=0$ otherwise. $\xrightarrow{a.s.}$ and $a.s.$ stand for almost sure convergence and almost surely. In addition, if two different measures appear, we write $\xrightarrow{\mathbb{P}-a.s.}$ and $\mathbb{E}_{\mathbb{P}}$ for almost sure convergence and expectation with respect to the measure~$\mathbb{P}$, respectively. Moreover, we say that an~event $A$ is true $\mathbb{P}-$a.s. if $\mathbb{P}(A)=1$.

%%%%%%%%%%%%%%%%%%%%%%%%%%%%%%%%%%%%%%%%%

\section{Preliminaries}
To state and prove the main result of this paper we
need to recall the concept of conditional quantile and collect some facts from the ergodic theory.

\begin{df}
Suppose $X$ is an~rv on a~probability space $(\Omega,\mathcal{F},\mathbb{P})$, $\mathcal{G}\subseteq\mathcal{F}$ is a~sigma-field and $\lambda\in(0,1)$. Then an~rv~$Q_{\lambda}$ with the following properties
\begin{description}
    \item[(i)]
    $Q_{\lambda}$ is $\mathcal{G}$-measurable,
    \item[(ii)]
    $\mathbb{P}(X\ge Q_{\lambda}|\mathcal{G})\ge 1-\lambda$ and $\mathbb{P}(X\le Q_{\lambda}|\mathcal{G})\ge\lambda$ a.s.
\end{description}
is called a~conditional $\lambda$th quantile of~$X$ with respect to~$\mathcal{G}$ and is denoted by $\pi_{\lambda}(X|\mathcal{G})$.
\end{df}

For every rv~$X$ and every sigma-field $\mathcal{G}$ there exists at least one conditional $\lambda$th quantile of~$X$ with respect to~$\mathcal{G}$. This statement is a~consequence of the same arguments as given in Tomkins for conditional medians (Theorem~1, 1975). Moreover,
we say that conditional $\lambda$th quantile of~$X$ with respect to~$\mathcal{G}$ is unique if given any two versions of $\pi_{\lambda}(X|\mathcal{G})$, $Q_{\lambda}$ and $Q_{\lambda}^{*}$ say, we have $Q_{\lambda}=Q_{\lambda}^{*}$ a.s. 
For more properties of conditional quantiles, we refer the reader to Tomkins (1975, 1978) and Ghosh and Mukherjee (2006).

Next, we recall some basic concepts and facts from the ergodic theory that we will use later on.
Let $(\mathbb{R}^{\mathbb{N}},\mathcal{B}(\mathbb{R}^{\mathbb{N}}),\mathbb{Q})$ denote a~probability triple, where $\mathbb{R}^{\mathbb{N}}$ is the set of sequences of real numbers $(x_1,x_2,\ldots)$,
$\mathcal{B}(\mathbb{R}^{\mathbb{N}})$ stands for the Borel sigma-field of subsets of $\mathbb{R}^{\mathbb{N}}$ and $\mathbb{Q}$ is a~stationary probability measure on the pair $(\mathbb{R}^{\mathbb{N}},\mathcal{B}(\mathbb{R}^{\mathbb{N}}))$.

\begin{df}
Let the transformation $T\colon \mathbb{R}^{\mathbb{N}}\to \mathbb{R}^{\mathbb{N}}$ be such that
\begin{equation}\label{transfT}
T((x_1,x_2,\ldots))=(x_2,x_3,\ldots).
\end{equation}
Then we call a set $B\in \mathcal{B}(\mathbb{R}^{\mathbb{N}})$
\begin{description}
 \item[(i)]
 invariant if $B=T^{-1}B$,
 \item[(ii)]
 almost invariant for $\mathbb{Q}$ if
 \[
 \mathbb{Q}((B\setminus T^{-1}B)\cup(T^{-1}B\setminus B))=0.
 \]
\end{description}
\end{df}
We write $\tilde{\mathcal{I}}$ and $\mathcal{I}_\mathbb{Q}$
for the class of all invariant events and the class of all almost invariant events for $\mathbb{Q}$, respectively.
Durrett (2010, Chapter 6) gave the following properties of $\tilde{\mathcal{I}}$ and $\mathcal{I}_\mathbb{Q}$.
\begin{lem}\label{l2}
\begin{description}
\item[(i)]
    $\tilde{\mathcal{I}}$ and $\mathcal{I}_\mathbb{Q}$ are sigma-fields.
\item[(ii)]
    An rv $X$ on $(\mathbb{R}^{\mathbb{N}},\mathcal{B}(\mathbb{R}^{\mathbb{N}}),\mathbb{Q})$ is $\tilde{\mathcal{I}}$-measurable (or $\mathcal{I}_\mathbb{Q}$-measurable) if and only if (iff)
    \begin{align*}
    X((x_1,x_2,\ldots))&=X((x_2,x_3,\ldots))\textrm{ for all }(x_1,x_2,\ldots)\in\mathbb{R}^{\mathbb{N}}\\
    \Big(\textrm{or } X((x_1,x_2,\ldots))&=X((x_2,x_3,\ldots))\textrm{ for }\mathbb{Q}\textrm{-almost every }(x_1,x_2,\ldots)\in\mathbb{R}^{\mathbb{N}}\Big).
    \end{align*}
\item[(iii)]
    If $B$ is almost invariant, then there exists an invariant set $C$ such that 
    \begin{equation}\label{BC}
    \mathbb{Q}((B\setminus C)\cup(C\setminus B))=0.
    \end{equation}
\end{description}
\end{lem}

We will also need the concept of invariant sets with respect to a~sequence of rv's.
\begin{df}
A~set $A\in\mathcal{F}$ is called invariant with respect to the sequence $\mathbb{X}=(X_n,n\ge 1)$ defined on~the probability space $(\Omega,\mathcal{F},\mathbb{P})$ if there exists a~set $B\in\mathcal{B}(\mathbb{R}^{\mathbb{N}})$ such that
\begin{equation}\label{warAinvar}
A=\{\omega\in\Omega:\ (X_i(\omega),X_{i+1}(\omega),\ldots)\in B \}\textrm{ for any }i\ge 1.
\end{equation}
\end{df}
The collection of all such invariant sets is denoted by $\mathcal{I}^{\mathbb{X}}$. 
Lemma~\ref{InvSeq} gives two basic properties of $\mathcal{I}^{\mathbb{X}}$. 
\begin{lem}\label{InvSeq}
Let $\mathbb{X}=(X_n,n\ge 1)$ be a~strictly stationary sequence on $(\Omega,\mathcal{F},\mathbb{P})$.
\begin{description}
 \item[(i)]
    $\mathcal{I}^{\mathbb{X}}$ is a~sigma-field.
  \item[(ii)]
    A~set $A\in\mathcal{F}$ is invariant with respect to~$\mathbb{X}$ iff there exists a~set $B\in\mathcal{I}_{\mathbb{Q}}$ satisfying~\eqref{warAinvar}, where the measure $\mathbb{Q}$ on $(\mathbb{R}^{\mathbb{N}},\mathcal{B}(\mathbb{R}^{\mathbb{N}}))$ is defined as follows
\begin{equation}\label{def_q}
\mathbb{Q}(B)=\mathbb{P}(\mathbb{X}\in B)\quad\textrm{for all }B\in\mathcal{B}(\mathbb{R}^{\mathbb{N}}).
\end{equation}
\end{description}
\end{lem}
For the proof of point~(i) of Lemma~\ref{InvSeq} we refer the reader to Shiryaev (1996), while point~(ii) can be found in Buraczy\'nska and Dembi\'nska (2018).

\section{New properties of conditional quantiles}
In this section we give some properties of conditional quantiles with respect to sigma-fields $\tilde{\mathcal{I}}$, $\mathcal{I}_{\mathbb{Q}}$ and~$\mathcal{I}^{\mathbb{X}}$.
\begin{tw}\label{dodatkowy2}
For any rv~$Y$ on the probability space $(\mathbb{R}^{\mathbb{N}},\mathcal{B}(\mathbb{R}^{\mathbb{N}}),\mathbb{Q})$, where $\mathbb{Q}$~is a~stationary probability measure, we have
\begin{description}
 \item[(i)]
 $\pi_{\lambda}(Y|\tilde{\mathcal{I}})$ is unique iff $\pi_{\lambda}(Y|\mathcal{I}_{\mathbb{Q}})$ is unique; 
 \item[(ii)]
 if $\pi_{\lambda}(Y|\tilde{\mathcal{I}})$ is unique (or equivalently $\pi_{\lambda}(Y|\mathcal{I}_{\mathbb{Q}})$ is unique) then
 \[
 \pi_{\lambda}(Y|\tilde{\mathcal{I}})=\pi_{\lambda}(Y|\mathcal{I}_{\mathbb{Q}})\quad\mathbb{Q}-\textrm{a.s.}
 \]
\end{description}
\end{tw}
In order to show Theorem~\ref{dodatkowy2}, we will need several facts, namely Lemmas~\ref{dodatkowy1}-\ref{pomocniczy2}. 

\begin{lem}\label{dodatkowy1}
For any rv $Y$ defined on the probability space $(\mathbb{R}^{\mathbb{N}},\mathcal{B}(\mathbb{R}^{\mathbb{N}}),\mathbb{Q})$ and such that $\mathbb{E}_{\mathbb{Q}}|Y|<\infty$, we have
\begin{equation}\label{teza3.2}
 \mathbb{E}_{\mathbb{Q}}(Y|\mathcal{I}_{\mathbb{Q}})
 = \mathbb{E}_{\mathbb{Q}}(Y|\tilde{\mathcal{I}})
 \quad\mathbb{Q}-\textrm{a.s.}
 \end{equation}
\end{lem}

\begin{proof}
To show \eqref{teza3.2}, we will prove that $\mathbb{E}_{\mathbb{Q}}(Y|\tilde{\mathcal{I}})$ 
is a~version of
$\mathbb{E}_{\mathbb{Q}}(Y|\mathcal{I}_{\mathbb{Q}})$.

Since $\tilde{\mathcal{I}}\subseteq\mathcal{I}_{\mathbb{Q}} $, $\mathbb{E}_{\mathbb{Q}}(Y|\tilde{\mathcal{I}})$ is $\mathcal{I}_{\mathbb{Q}}$-measurable. Additionaly, we must show that
\begin{equation}\label{d1}
\forall_{B\in\mathcal{I}_{\mathbb{Q}}}\quad \mathbb{E}_{\mathbb{Q}}(\mathbb{E}_{\mathbb{Q}}(Y|\tilde{\mathcal{I}})I(B))= \mathbb{E}_{\mathbb{Q}}(Y I(B)).
\end{equation}
Observe that, by the definition of $\mathbb{E}_{\mathbb{Q}}(Y|\tilde{\mathcal{I}})$, $\mathbb{E}_{\mathbb{Q}}(Y|\tilde{\mathcal{I}})$ is $\tilde{\mathcal{I}}$-measurable and
\begin{equation}\label{d2}
\forall_{C\in\tilde{\mathcal{I}}}\quad \mathbb{E}_{\mathbb{Q}}(\mathbb{E}_{\mathbb{Q}}(Y|\tilde{\mathcal{I}})I(C))= \mathbb{E}_{\mathbb{Q}}(Y I(C)).
\end{equation}
Moreover, by Lemma~\ref{l2}~(iii), for every almost invariant set $B$, there exists an invariant set~$C$ such that~\eqref{BC} is fulfilled. Therefore, 
\begin{equation}
    \forall_{B\in\mathcal{I}_{\mathbb{Q}}} \exists_{C\in\tilde{\mathcal{I}}}\quad\mathbb{E}_{\mathbb{Q}}(\mathbb{E}_{\mathbb{Q}}(Y|\tilde{\mathcal{I}})I(B))
    = \mathbb{E}_{\mathbb{Q}}(\mathbb{E}_{\mathbb{Q}}(Y|\tilde{\mathcal{I}})I(C))=
    \mathbb{E}_{\mathbb{Q}}(Y I(C))
    = \mathbb{E}_{\mathbb{Q}}(Y I(B)),    
\end{equation}
where the second equality is a~consequence of~\eqref{d2}. This establishes~\eqref{d1}. So $\mathbb{E}_{\mathbb{Q}}(Y|\tilde{\mathcal{I}})$ is a~conditional expectation of the rv~$Y$ with respect to sigma-field~$\mathcal{I}_{\mathbb{Q}}$.
\end{proof}

\begin{lem}\label{pomocniczy1}
If $Q_{\lambda}$ is a~version of $\pi_{\lambda}(Y|\tilde{\mathcal{I}})$ then $Q_{\lambda}$ is also a~version of $\pi_{\lambda}(Y|\mathcal{I}_{\mathbb{Q}})$.
\end{lem}

\begin{proof}
Fist observe that by Lemma~\ref{dodatkowy1}, for any rv $Q_{\lambda}$ on the probability space $(\mathbb{R}^{\mathbb{N}},\mathcal{B}(\mathbb{R}^{\mathbb{N}}),\mathbb{Q})$, we have
\begin{equation}\label{d3}
    \mathbb{Q}(Y\ge Q_{\lambda}|\mathcal{I}_{\mathbb{Q}})\ge 1-\lambda \textrm{ and }
    \mathbb{Q}(Y\le Q_{\lambda}|\mathcal{I}_{\mathbb{Q}})\ge \lambda\quad \mathbb{Q}-\textrm{a.s.}
\end{equation}
iff
\begin{equation}\label{d4}
    \mathbb{Q}(Y\ge Q_{\lambda}|\tilde{\mathcal{I}})\ge 1-\lambda \textrm{ and }
    \mathbb{Q}(Y\le Q_{\lambda}|\tilde{\mathcal{I}})\ge \lambda\quad \mathbb{Q}-\textrm{a.s.}
\end{equation}

Since $Q_{\lambda}$ is a~version of $\pi_{\lambda}(Y|\tilde{\mathcal{I}})$, we get that
\begin{enumerate}
    \item 
    $Q_{\lambda}$ is $\tilde{\mathcal{I}}$-measurable,
    \item
    $Q_{\lambda}$ satisfies~\eqref{d4}.
\end{enumerate}

As $\tilde{\mathcal{I}}\subseteq\mathcal{I}_{\mathbb{Q}}$, we conclude that $Q_{\lambda}$ is $\mathcal{I}_{\mathbb{Q}}$-measurable. Moreover, \eqref{d4} implies~\eqref{d3} in view of the previous observation. This completes the proof.
\end{proof}

\begin{lem}\label{pomocniczy2}
If $Q_{\lambda}$ is a~version of $\pi_{\lambda}(Y|\mathcal{I}_{\mathbb{Q}})$ then there exists an~rv $R_{\lambda}$ such that $R_{\lambda}$ is a~version of $\pi_{\lambda}(Y|\tilde{\mathcal{I}})$ and $Q_{\lambda}=R_{\lambda}$ $\mathbb{Q}$-a.s.
\end{lem}

\begin{proof}
Let $Q_{\lambda}$ be a~version of $\pi_{\lambda}(Y|\mathcal{I}_{\mathbb{Q}})$. Then~\eqref{d3} holds so in consequence~\eqref{d4} is also fulfilled.
Since $Q_{\lambda}$ is $\mathcal{I}_{\mathbb{Q}}$-measurable, Condition (ii) of Lemma~\ref{l2} gives, for any $i\ge 1$,
\begin{equation}\label{d5} Q_{\lambda}((x_1,x_2,\ldots))=Q_{\lambda}((x_i,x_{i+1},\ldots))\textrm{ for }\mathbb{Q}-\textrm{almost every }(x_1,x_2,\ldots)\in\mathbb{R}^{\mathbb{N}}.
\end{equation}

Let us define
\[
R_{\lambda}((x_1,x_2,\ldots)) = Q_{\lambda}((x_{i_0},x_{i_0+1},\ldots))
\]
if there exists $i_0$ such that $Q_{\lambda}((x_i,x_{i+1},\ldots))=Q_{\lambda}((x_{i+1},x_{i+2},\ldots))$ for all $i\ge i_0$ and $R_{\lambda}((x_1,x_2,\ldots))=0$ otherwise. Then
\[
R_{\lambda}((x_1,x_2,\ldots))=R_{\lambda}((x_2,x_{3},\ldots))\textrm{ for every }(x_1,x_2,\ldots)\in\mathbb{R}^{\mathbb{N}}
\]
and by Lemma~\ref{l2} (ii),
\begin{equation}\label{d6}
    R_{\lambda}\textrm{ is }\tilde{\mathcal{I}}-\textrm{measurable}.
\end{equation}
Next observe that $R_{\lambda}=Q_{\lambda}$ $\mathbb{Q}-$a.s., because
\begin{align*}
\{(x_1,x_2,\ldots)\in\mathbb{R}^{\mathbb{N}}\colon R_{\lambda}((x_1,x_2,\ldots))=Q_{\lambda}((x_1,x_2,\ldots))\}\\
\supseteq
\{(x_1,x_2,\ldots)\in\mathbb{R}^{\mathbb{N}}\colon Q_{\lambda}((x_1,x_2,\ldots))=Q_{\lambda}((x_i,x_{i+1},\ldots))\textrm{ for all }i\ge 1\}
\end{align*}
and
\begin{align*}
    \mathbb{Q}(
    \{(x_1,x_2,\ldots)\in\mathbb{R}^{\mathbb{N}}\colon Q_{\lambda}((x_1,x_2,\ldots))=Q_{\lambda}((x_i,x_{i+1},\ldots))\textrm{ for all }i\ge 1\})=1
\end{align*}
by~\eqref{d5}.

Since $R_{\lambda}=Q_{\lambda}$ $\mathbb{Q}-$a.s., \eqref{d4} implies
\begin{equation}\label{d7}
    \mathbb{Q}(Y\ge R_{\lambda}|\tilde{\mathcal{I}})\ge 1-\lambda\textrm{ and }
    \mathbb{Q}(Y\le R_{\lambda}|\tilde{\mathcal{I}})\ge \lambda\quad \mathbb{Q}-\textrm{a.s.}
\end{equation}
By~\eqref{d6} and~\eqref{d7}, $R_{\lambda}$ is a~version of $\pi_{\lambda}(Y|\tilde{\mathcal{I}})$. This completes the proof.
\end{proof}

Now, using Lemmas~\ref{dodatkowy1}-\ref{pomocniczy2},
we can prove Theorem~\ref{dodatkowy2}.
\begin{proof}[Proof of Theorem~\ref{dodatkowy2}]
Let us assume first that $\pi_{\lambda}(Y|\tilde{\mathcal{I}})$ is not unique. This means that there exist two versions of $\pi_{\lambda}(Y|\tilde{\mathcal{I}})$, $Q_{\lambda}$ and $Q_{\lambda}^{*}$ say, such that
\[
\mathbb{Q}(\{(x_1,x_2,\ldots)\in\mathbb{R}^{\mathbb{N}}\colon Q_{\lambda}((x_1,x_2,\ldots))\neq Q_{\lambda}^{*}((x_1,x_2,\ldots))\})>0.
\]
By Lemma~\ref{pomocniczy1}, any version of~$\pi_{\lambda}(Y|\tilde{\mathcal{I}})$ is a~version of $\pi_{\lambda}(Y|\mathcal{I}_{\mathbb{Q}})$, so we get that there exist two versions of $\pi_{\lambda}(Y|\mathcal{I}_{\mathbb{Q}})$, $Q_{\lambda}$ and $Q_{\lambda}^{*}$, which are not $\mathbb{Q}-$a.s. equal. Thus, $\pi_{\lambda}(Y|\mathcal{I}_{\mathbb{Q}})$ is not unique.
Therefore, we have proved that
\begin{equation}\label{d7A}
    \pi_{\lambda}(Y|\tilde{\mathcal{I}})\textrm{ is not unique } \Rightarrow \pi_{\lambda}(Y|\mathcal{I}_{\mathbb{Q}})\textrm{ is not unique.}
\end{equation}

Now assume that $\pi_{\lambda}(Y|\mathcal{I}_{\mathbb{Q}})$ is not unique. Then there exist two versions of $\pi_{\lambda}(Y|\mathcal{I}_{\mathbb{Q}})$, $\hat{Q}_{\lambda}$ and $\hat{Q}_{\lambda}^{*}$ say, such that
\begin{equation}\label{d8}
\mathbb{Q}(\{(x_1,x_2,\ldots)\in\mathbb{R}^{\mathbb{N}}\colon \hat{Q}_{\lambda}((x_1,x_2,\ldots))\neq \hat{Q}_{\lambda}^{*}((x_1,x_2,\ldots))\})>0.
\end{equation}
By Lemma~\ref{pomocniczy2}, there exist two versions of $\pi_{\lambda}(Y|\tilde{\mathcal{I}})$, $R_{\lambda}$ and $R_{\lambda}^{*}$ say, such that
\[
R_{\lambda}=\hat{Q}_{\lambda}\textrm{ and }R_{\lambda}^{*}=\hat{Q}_{\lambda}^{*}\quad\mathbb{Q-}\textrm{a.s.}
\]
In consequence of \eqref{d8}, we have
\[
\mathbb{Q}(\{(x_1,x_2,\ldots)\in\mathbb{R}^{\mathbb{N}}\colon R_{\lambda}((x_1,x_2,\ldots))\neq R_{\lambda}^{*}((x_1,x_2,\ldots))\})>0.
\]
Thus we have showed that $\pi_{\lambda}(Y|\tilde{\mathcal{I}})$ is not unique and consequently that
\begin{equation}\label{d8A}
    \pi_{\lambda}(Y|\mathcal{I}_{\mathbb{Q}})\textrm{ is not unique}\Rightarrow
    \pi_{\lambda}(Y|\tilde{\mathcal{I}})\textrm{ is not unique.}
\end{equation}
But~\eqref{d7A} combined with~\eqref{d8A} is equivalent to part (i) of Theorem~\ref{dodatkowy2}.

Next, we move to part (ii) of this theorem. Let $Q_{\lambda}$ be a~version of~$\pi_{\lambda}(Y|\tilde{\mathcal{I}})$. Then, by Lemma~\ref{pomocniczy1}, $Q_{\lambda}$ is also a~version of~$\pi_{\lambda}(Y|\mathcal{I}_{\mathbb{Q}})$. Since all versions of $\pi_{\lambda}(Y|\tilde{\mathcal{I}})$ are $\mathbb{Q}-$a.s. equal and all versions of~$\pi_{\lambda}(Y|\mathcal{I}_{\mathbb{Q}})$ are $\mathbb{Q}-$a.s. equal, we get
\[
\pi_{\lambda}(Y|\tilde{\mathcal{I}})=Q_{\lambda}=\pi_{\lambda}(Y|\mathcal{I}_{\mathbb{Q}})\quad\mathbb{Q-}a.s.
\]
This is precisely the assertion of part~(ii).
\end{proof}

The second crucial result of this section is Theorem~\ref{fakt3}.
We begin by formulating Lemmas~\ref{fakt}-\ref{fakt2}. They will be the basis to prove Theorem~\ref{fakt3}.
\begin{lem}\label{fakt}
Let $Q$ be an rv on a~probability space $(\mathbb{R}^{\mathbb{N}},\mathcal{\mathbb{R}^{\mathbb{N}}},\mathbb{Q})$, where $\mathbb{Q}$ is a~stationary probability measure, and $(X_n,n\ge 1)$ be a~strictly stationary sequence on some probability space $(\Omega,\mathcal{F},\mathbb{P})$. If the rv $Q$ is $\tilde{\mathcal{I}}$-measurable, then the rv $Q((X_1,X_2,\ldots))$ is $\mathcal{I}^{\mathbb{X}}$-measurable.
\end{lem}

\begin{proof}
Our aim is to show that
\[
\{\omega\colon Q((X_1(\omega),X_2(\omega),\ldots))\in A\}\in \mathcal{I}^{\mathbb{X}}\quad\textrm{for all }A\in\mathcal{B}(\mathbb{R}).
\]
By the definition of $\mathcal{I}^{\mathbb{X}}$, this means the same as
\begin{align*}
    &\textrm{for all }A\in\mathcal{B}(\mathbb{R})\textrm{ there exists }B\in\mathcal{B}(\mathbb{R}^{\mathbb{N}})\textrm{ such that}\\
    \{\omega\colon Q((X_1(\omega),& X_2(\omega),\ldots))\in A\}=
    \{\omega\colon (X_i(\omega),X_{i+1}(\omega),\ldots)\in B\}\textrm{ for any }i\ge 1.
\end{align*}

Using the assumption that $Q$ is $\tilde{\mathcal{I}}$- measurable yields
\[
\{(x_1,x_2,\ldots)\in\mathcal{B}(\mathbb{R}^{\mathbb{N}})\colon Q((x_1,x_2,\ldots))\in A\}\in \tilde{\mathcal{I}}\quad\textrm{for all }A\in\mathcal{B}(\mathbb{R}),
\]
but this ensures that
\begin{align}\label{d10}
    \textrm{for all }A\in\mathcal{B}(\mathbb{R})\textrm{ there exists }B\in\tilde{\mathcal{I}}\subset\mathcal{B}(\mathbb{R}^{\mathbb{N}})\textrm{ such that}\nonumber\\
    \{(x_1,x_2,\ldots)\in\mathcal{B}(\mathbb{R}^{\mathbb{N}})\colon Q((x_1,x_2,\ldots))\in A\}=B.
\end{align}
Therefore, for all $A\in\mathcal{B}(\mathbb{R})$ there exists $B\in\tilde{\mathcal{I}}\subset\mathcal{B}(\mathbb{R}^{\mathbb{N}})$ such that
\begin{align*}
    \{\omega\colon Q((X_1(\omega),X_2(\omega),\ldots))\in A\}=
    \{\omega\colon (X_1(\omega),X_2(\omega),\ldots)\in B\}\\
    =\{\omega\colon (X_i(\omega),X_{i+1}(\omega),\ldots)\in B\}\textrm{ for any }i\ge 1,
\end{align*}
and the proof is complete.
\end{proof}

\begin{lem}\label{fakt2}
Let $\mathbb{X}=(X_n,n\ge 1)$ be a~strictly stationary sequence on a~probability space $(\Omega,\mathcal{F},\mathbb{P})$ and $\mathbb{Q}$ be the probability measure on $(\mathbb{R}^{\mathbb{N}},\mathcal{B}(\mathbb{R}^{\mathbb{N}}))$ defined by~\eqref{def_q}.
%\begin{equation}\label{def_q}
%\mathbb{Q}(B)=\mathbb{P}(\mathbb{X}\in B)\quad\textrm{for all %}B\in\mathcal{B}(\mathbb{R}^{\mathbb{N}}).
%\end{equation}
Then, for any set $D\in\mathcal{B}(\mathbb{R}^{\mathbb{N}})$,
\[
\mathbb{E}_{\mathbb{P}}(I(D)((X_1,X_2,\ldots))|\mathcal{I}^{\mathbb{X}})=
\mathbb{E}_{\mathbb{Q}}(I(D)|\tilde{\mathcal{I}})((X_1,X_2,\ldots))\quad\mathbb{P}-\textrm{a.s.}
\]
\end{lem}

\begin{proof}
Our aim is to show that $\mathbb{E}_{\mathbb{Q}}(I(D)|\tilde{\mathcal{I}})((X_1,X_2,\ldots))$ is a~version of the conditional expectation of the rv $I(D)((X_1,X_2,\ldots))$ with respect to the sigma-field $\mathcal{I}^{\mathbb{X}}$.

First note that $\mathbb{E}_{\mathbb{Q}}(I(D)|\tilde{\mathcal{I}})$ is $\tilde{\mathcal{I}}$-measurable, so
$\mathbb{E}_{\mathbb{Q}}(I(D)|\tilde{\mathcal{I}})((X_1,X_2,\ldots))$ is $\mathcal{I}^{\mathbb{X}}$-measurable in view of Lemma~\ref{fakt}.
Thus, we are reduced to proving that
\begin{equation}\label{a1}
    \mathbb{E}_{\mathbb{P}}\Big(I(D) ((X_1,X_2,\ldots)) I(A)\Big)
    = \mathbb{E}_{\mathbb{P}}\Big(\mathbb{E}_{\mathbb{Q}}(I(D)|\tilde{\mathcal{I}}) ((X_1,X_2,\ldots)) I(A)\Big)
    \textrm{ for all }A\in\mathcal{I}^{\mathbb{X}}.
\end{equation}

Fix $A\in\mathcal{I}^{\mathbb{X}}$ and observe that, for all $i\ge 1$,
\begin{align*}
    \mathbb{E}_{\mathbb{P}}(I(D) ((X_1,X_2,\ldots)) I(A))
    = \mathbb{E}_{\mathbb{P}}(I((X_1,X_2,\ldots)\in D) I(A))\\
    = \mathbb{P}(\{(X_1,X_2,\ldots)\in D\}\cap \{(X_i,X_{i+1},\ldots)\in B\}),
\end{align*}    
where the set $B\in\mathcal{B}(\mathbb{R}^{\mathbb{N}})$ is such that \eqref{warAinvar} holds. In particular,
\begin{align*}
    \mathbb{E}_{\mathbb{P}}(I(D)((X_1,X_2,\ldots))I(A))=\mathbb{P}((X_1,X_2,\ldots)\in D\cap B)
    =\mathbb{Q}(D\cap B),
\end{align*}
by \eqref{def_q}. From part~(iii) of Lemma~\ref{l2}, there exists an~invariant set $E$ such that
\begin{equation}\label{BE}
\mathbb{Q}((B\setminus E)\cup(E\setminus B))=0,
\end{equation}
so in consequence we have
\begin{equation}\label{a2}
\mathbb{E}_{\mathbb{P}}(I(D)((X_1,X_2,\ldots)) I(A)) =
\mathbb{Q}(D\cap B) = \mathbb{Q}(D\cap E).
\end{equation}

On the other hand, for any $i\ge 1$,
\begin{align*}
    \mathbb{E}_{\mathbb{P}}&\Big(\mathbb{E}_{\mathbb{Q}}(I(D)|\tilde{\mathcal{I}}) ((X_1,X_2,\ldots)) I(A)\Big)\nonumber\\
    &= \mathbb{E}_{\mathbb{P}}\Big(\mathbb{E}_{\mathbb{Q}}(I(D)|\tilde{\mathcal{I}}) ((X_1,X_2,\ldots))
    I((X_i,X_{i+1},\ldots)\in B)\Big)\nonumber\\ 
    &= \mathbb{E}_{\mathbb{P}}\Big(\mathbb{E}_{\mathbb{Q}}(I(D)|\tilde{\mathcal{I}}) ((X_1,X_2,\ldots))
    I(B)((X_1,X_{2},\ldots))\Big)\nonumber\\
    &= \mathbb{E}_{\mathbb{P}}\Big((\mathbb{E}_{\mathbb{Q}}(I(D)|\tilde{\mathcal{I}})I(B)) ((X_1,X_2,\ldots))\Big)
    = \mathbb{E}_{\mathbb{Q}}(\mathbb{E}_{\mathbb{Q}}(I(D)|\tilde{\mathcal{I}})I(B)),
\end{align*}
where the last equality is a~consequence of \eqref{def_q}.
Using \eqref{BE} yields
\begin{align}\label{a3}    
    \mathbb{E}_{\mathbb{Q}}(\mathbb{E}_{\mathbb{Q}}(I(D)|\tilde{\mathcal{I}})
    I(B))= \mathbb{E}_{\mathbb{Q}}(\mathbb{E}_{\mathbb{Q}}(I(D)|\tilde{\mathcal{I}})
    I(E))\nonumber\\
    = \mathbb{E}_{\mathbb{Q}}(\mathbb{E}_{\mathbb{Q}}(I(D) \cap I(E)|\tilde{\mathcal{I}}))
    =\mathbb{E}_{\mathbb{Q}}(I(D)\cap I(E))
    =\mathbb{Q}(D\cap E).
\end{align}
Combining \eqref{a2} with~\eqref{a3} we obtain \eqref{a1} and the proof is complete.
\end{proof}

\begin{tw}\label{fakt3}
Let $\mathbb{X}=(X_n,n\ge 1)$ be a~strictly stationary sequence of rv's on any probability space $(\Omega,\mathcal{F},\mathbb{P})$ and let $\mathbb{Q}$ be a~stationary measure on $(\mathbb{R}^{\mathbb{N}},\mathcal{B}(\mathbb{R}^{\mathbb{N}}))$ defined by~\eqref{def_q}.
On $(\mathbb{R}^{\mathbb{N}},\mathcal{B}(\mathbb{R}^{\mathbb{N}}),\mathbb{Q})$ we define the rv $Y\colon\mathbb{R}^{\mathbb{N}}\to\mathbb{R}$,
\begin{equation}\label{def_Y1}
Y((x_1,x_2,\ldots))=x_1\quad\textrm{for }(x_1,x_2,\ldots)\in\mathbb{R}^{\mathbb{N}}.
\end{equation}
If the conditional $\lambda$th quantile $\pi_{\lambda}(X_1|\mathcal{I}^{\mathbb{X}})$ of~$X_1$ given~$\mathcal{I}^{\mathbb{X}}$ is unique then the conditional $\lambda$th quantile $\pi_{\lambda}(Y|\tilde{\mathcal{I})}$ of~$Y$ given~$\tilde{\mathcal{I}}$ is also unique.
\end{tw}
\begin{proof}%[Proof of Theorem~\ref{fakt3}]
First we will show that if~$R_{\lambda}$ is a~version of $\pi_{\lambda}(Y|\tilde{\mathcal{I})}$ then $R_{\lambda}((X_1,X_2,\ldots))$ is a~version of $\pi_{\lambda}(X_1|\mathcal{I}^{\mathbb{X}})$. Since $R_{\lambda}$ is a version of $\pi_{\lambda}(Y|\tilde{\mathcal{I})}$, the following conditions hold
\begin{enumerate}
    \item 
    $R_{\lambda}$ is $\tilde{\mathcal{I}}$-measurable;
    \item
    $\mathbb{Q}(Y\ge R_{\lambda}|\tilde{\mathcal{I}})\ge 1-\lambda$ and $\mathbb{Q}(Y\le R_{\lambda}|\tilde{\mathcal{I}})\ge\lambda\quad\mathbb{Q-}$a.s.
\end{enumerate}
Observe that $R_{\lambda}((X_1,X_2,\ldots))$ is $\mathcal{I}^{\mathbb{X}}$-measurable in consequence of Condition~1 and Lemma~\ref{fakt}. Next we will show that
\[
\mathbb{P}(X_1\ge R_{\lambda}((X_1,X_2,\ldots))|\mathcal{I}^{\mathbb{X}})\ge 1-\lambda\quad\mathbb{P}-\textrm{a.s.}
\]
Indeed,
\begin{align*}
    \mathbb{P}(&X_1\ge R_{\lambda}((X_1,X_2,\ldots))|\mathcal{I}^{\mathbb{X}})
    =\mathbb{P}(Y((X_1,X_2,\ldots))\ge R_{\lambda}((X_1,X_2,\ldots))|\mathcal{I}^{\mathbb{X}})\\
    &=\mathbb{E}_{\mathbb{P}}\Big(I(Y\ge R_{\lambda})((X_1,X_2,\ldots))|\mathcal{I}^{\mathbb{X}}\Big)\\
    &=\mathbb{E}_{\mathbb{Q}}(I(Y\ge R_{\lambda})|\tilde{\mathcal{I}})((X_1,X_2,\ldots))
     %=\mathbb{Q}(Y\ge R_{\lambda})|\tilde{\mathcal{I}})((X_1,X_2,\ldots))\ge 1-\lambda
     \quad\mathbb{P}-a.s.,
\end{align*}
where the first equality and the third one are consequences of the definition of~$Y$ and~Lemma~\ref{fakt2}, respectively. But
$\mathbb{E}_{\mathbb{Q}}(I(Y\ge R_{\lambda})|\tilde{\mathcal{I}})((X_1,X_2,\ldots))\ge 1-\lambda$ $\mathbb{P}$-a.s. by~\eqref{def_q} and Condition~2, because
\[
\mathbb{P}(\mathbb{E}_{\mathbb{Q}}(I(Y\ge R_{\lambda})|\tilde{\mathcal{I}})((X_1,X_2,\ldots))\ge 1-\lambda)
= \mathbb{Q}(\mathbb{E}_{\mathbb{Q}}(I(Y\ge R_{\lambda})|\tilde{\mathcal{I}}) \ge 1-\lambda)=1.
\]
Thus we have proved that
\begin{align*}
    \mathbb{P}(X_1\ge R_{\lambda}((X_1,X_2,\ldots))|\mathcal{I}^{\mathbb{X}})\ge 1-\lambda\quad\mathbb{P-}\textrm{a.s.}
\end{align*}
In the same manner we can show that
\begin{align*}
    \mathbb{P}(X_1\le R_{\lambda}((X_1,X_2,\ldots))|\mathcal{I}^{\mathbb{X}})\ge \lambda\quad\mathbb{P-}\textrm{a.s.}
\end{align*}
Therefore $R_{\lambda}((X_1,X_2,\ldots))$ is a~version of $\pi_{\lambda}(X_1|\mathcal{I}^{\mathbb{X}})$.

Now we move to the next part of the proof.
Let $\pi_{\lambda}(Y|\tilde{\mathcal{I}})$ be not unique. It means that we have two versions of 
$\pi_{\lambda}(Y|\tilde{\mathcal{I}})$, $Q_{\lambda}$ and $Q_{\lambda}^{*}$ say, such that
\begin{equation}\label{n1}
\mathbb{Q}(\{(x_1,x_2,\ldots)\in\mathbb{R}^{\mathbb{N}}\colon Q_{\lambda}((x_1,x_2,\ldots))\neq Q_{\lambda}^{*}((x_1,x_2,\ldots))\})>0.
\end{equation}
By the previous observation, we conclude that $Q_{\lambda}((X_1,X_2,\ldots))$ and $Q_{\lambda}^{*}((X_2,X_2,\ldots))$ are two versions of 
$\pi_{\lambda}(X_1|\mathcal{I}^{\mathbb{X}})$.
Moreover,
\[
\mathbb{P}(Q_{\lambda}((X_1,X_2,\ldots))\neq Q_{\lambda}^{*}((X_1,X_2,\ldots)))
=\mathbb{Q}(Q_{\lambda}\neq Q_{\lambda}^{*})>0
\]
by~\eqref{def_q} and~\eqref{n1}.
This proves that $\pi_{\lambda}(X_1|\mathcal{I}^{\mathbb{X}})$ is not unique.
\end{proof}

\section{New version of strong ergodic theorem for central order statistics}
Our purpose is to investigate the almost sure convergence of central order statistics. Definitely, the most general result dealing with the problem was derived by Dembi\'nska (2014). 
Her crucial result provides conditions that are sufficient for the existence of the almost sure limit of central order statistics arising from a strictly stationary process $\mathbb{X}=(X_n,n\ge 1)$, and describes the distribution of the limiting rv. The main idea of her proof is to first show the corresponding result for a~sequence $\mathbb{Y}=(Y_n,n\ge 1)$ of rv's from the probability space $(\mathbb{R}^{\mathbb{N}},\mathcal{B}(\mathbb{R}^{\mathbb{N}}),\mathbb{Q})$ and then to generalize this result to an~arbitrary strictly stationary processes existing in any probability triple $(\Omega,\mathcal{F},\mathbb{P})$.
We begin by recalling the above-mentioned facts (Theorems~\ref{tw.1}-\ref{tw.6}).

\begin{tw}\label{tw.1}
Let $Y$ be an rv on a probability space $(\mathbb{R}^{\mathbb{N}},\mathcal{B}(\mathbb{R}^{\mathbb{N}}),\mathbb{Q})$, where the probability
measure $\mathbb{Q}$ is stationary. Suppose that the sequence of rv's $(Y_n, n \ge 1)$ is defined by
\begin{equation}\label{def_Y}
Y_i ((x_1, x_2,\ldots)) = Y ((x_i , x_{i+1},\ldots))\textrm{ for }(x_1, x_2,\ldots) \in \mathbb{R}^{\mathbb{N}}\textrm{ and }i \ge 1.
\end{equation}
If $(k_n, n \ge 1)$ is a~sequence of  integers satisfying~\eqref{warK}
and if the conditional $\lambda$th quantile $\pi_{\lambda}(Y|\mathcal{I}_{\mathbb{Q}})$ of~$Y$ given~$\mathcal{I}_{\mathbb{Q}}$ is unique, then
\begin{equation}\label{teza1}
Y_{k_n:n}\xrightarrow{\mathbb{Q}-a.s.}\pi_{\lambda}(Y|\mathcal{I}_{\mathbb{Q}}).
\end{equation}
\end{tw}

\begin{tw}\label{tw.6}
Let the sequence $\mathbb{X}=(X_n,n\ge 1)$, measure $\mathbb{Q}$ and rv~$Y$ be as in Theorem~\ref{fakt3}.
If $(k_n,n\ge 1)$ is a~sequence satisfying~\eqref{warK} and
the conditional $\lambda$th quantile $\pi_{\lambda}(Y|\mathcal{I}_{\mathbb{Q}})$ of~$Y$ given~$\mathcal{I}_{\mathbb{Q}}$ is unique then there exists an~rv~$W$ such that
\begin{equation}\label{gwiazdka}
X_{k_n:n}\xrightarrow{\mathbb{P}-a.s.}W\textrm{ as }n\to\infty.
\end{equation}
Moreover the rv $W$ has the same distribution as 
$\pi_{\lambda}(Y|\mathcal{I}_{\mathbb{Q}})$.
\end{tw}

Even though Theorem~\ref{tw.6} concerns central order statistics from the sequence of rv's $(X_n,n\ge 1)$, existing in a~probability space 
$(\Omega,\mathcal{F},\mathbb{P})$, in its formulation the rv~$Y$ from the probability triple $(\mathbb{R}^{\mathbb{N}},\mathcal{B}(\mathbb{R}^{\mathbb{N}}),\mathbb{Q})$ is used. The main result of this paper is the following elegant version of Theorem~\ref{tw.6}, in which $Y$ does not appear -- the assumptions and conclusion are given only in terms of the sequence $(X_n,n\ge 1)$.
\begin{tw}\label{tw.2}
Let $\mathbb{X}=(X_n,n\ge 1)$ be a~strictly stationary sequence and $(k_n,n\ge 1)$ be a~sequence of integers satisfying~\eqref{warK}. If the conditional $\lambda$th quantile $\pi_{\lambda}(X_1|\mathcal{I}^{\mathbb{X}})$ of~$X_1$ given~$\mathcal{I}^{\mathbb{X}}$ is unique, then
\begin{equation}%\label{teza2}
X_{k_n:n}\xrightarrow{\mathbb{P}-a.s.}\pi_{\lambda}(X_1|\mathcal{I}^{\mathbb{X}})\textrm{ as }n\to\infty.
\end{equation}
\end{tw}
We first prove a reduced form of Theorem~\ref{tw.2}, namely Lemma~\ref{lem4.2} and then using Theorems~\ref{dodatkowy1}-\ref{fakt3} we will prove the main result.

\begin{lem}\label{lem4.2}
Under the assumptions of Theorem~\ref{tw.6},
\begin{equation}\label{teza2}
X_{k_n:n}\xrightarrow{\mathbb{P}-a.s.}\pi_{\lambda}(X_1|\mathcal{I}^{\mathbb{X}})\textrm{ as }n\to\infty.
\end{equation}
\end{lem}

\begin{proof}
The conclusion of Theorem~\ref{tw.6} ensures that $\lim_{n\to\infty} X_{k_n:n}$ exists $\mathbb{P}$-a.s. To define the value of this limit  on the set (of $\mathbb{P}$-probability zero) of $\omega\in\Omega$ such that $\lim_{n\to\infty}X_{k_n:n}(\omega)$ does not exist, let us assume, for instance, 
\[
\lim_{n\to\infty}X_{k_n:n}(\omega)=0.
\]
Showing~\eqref{teza2} amounts to proving the following three conditions:
\begin{enumerate}
\item
$\lim_{n\to\infty}X_{k_n:n}$ is $\mathcal{I}^{\mathbb{X}}$-measurable,
\item
$\mathbb{P}(X_1\ge \lim_{n\to\infty}X_{k_n:n}|\mathcal{I}^{\mathbb{X}})\ge 1-\lambda$ $\mathbb{P}$-a.s.,
\item
$\mathbb{P}(X_1\le \lim_{n\to\infty}X_{k_n:n}|\mathcal{I}^{\mathbb{X}})\ge \lambda$ $\mathbb{P}$-a.s.
\end{enumerate}

Condition~1 means that for all $A\in\mathcal{B}(\mathbb{R})$,
$$\{\omega\in\Omega:\ \lim_{n\to\infty}X_{k_n:n}(\omega)\in A\}\in \mathcal{I}^{\mathbb{X}},$$
which, by the definition of $\mathcal{I}^{\mathbb{X}}$, corresponds to the following statement:
\begin{align*}
\textrm{for all } A\in\mathcal{B}(\mathbb{R}) \textrm{ there exists } B\in \mathcal{B}(\mathbb{R}^{\mathbb{N}})  \textrm{ such that for all  } n\ge 1 \\
\{\omega\in\Omega:\ \lim_{n\to\infty}X_{k_n:n}(\omega)\in A\}=\{\omega\in\Omega:\ (X_n(\omega),X_{n+1}(\omega),\ldots)\in B\}.
\end{align*}
Let us consider  $B=\{(x_1,x_2,\ldots)\in\mathbb{R}^{\mathbb{N}}:\ \lim_{n\to\infty} x_{k_n:n}\in A\}$, where $\lim_{n\to\infty}x_{k_n:n}$ is defined to equal $0$ if this limit does not exist. Note that by Theorem~\ref{dodatkowy1}, Theorem~\ref{tw.1} still holds if \eqref{teza1} is replaced by
\begin{equation}\label{inna_teza}
Y_{k_n:n}\xrightarrow{\mathbb{Q}-a.s.}\pi_{\lambda}(Y|\tilde{\mathcal{I}})    
\end{equation}
Consequently $B\in\tilde{\mathcal{I}}$, so it is clear that $B\in \mathcal{B}(\mathbb{R}^{\mathbb{N}})$ and for all $n\ge 1$,
\[
\{\omega\in\Omega:\ (X_1(\omega),X_{2}(\omega),\ldots)\in B\}=\{\omega\in\Omega:\ (X_n(\omega),X_{n+1}(\omega),\ldots)\in B\}.
\]
This completes part~1 of the proof.

In order to prove condition~2, first note that \eqref{inna_teza} ensures that $\lim_{n\to\infty}Y_{k_n:n}$ is $\tilde{\mathcal{I}}$-measurable and
\begin{equation}\label{pomoc}
\mathbb{Q}(Y\ge\lim_{n\to\infty}Y_{k_n:n}|\tilde{\mathcal{I}})\ge 1-\lambda
\textrm{ and }
\mathbb{Q}(Y\le \lim_{n\to\infty}Y_{k_n:n}|\tilde{\mathcal{I}})\ge \lambda\quad \mathbb{Q}-\textrm{a.s.}
\end{equation}
Furthermore,
\begin{align*}
\mathbb{P}(X_1&\ge\lim_{n\to\infty}X_{k_n:n}|\mathcal{I}^{\mathbb{X}})= \mathbb{P}\Big(Y((X_1,X_2,\ldots))\ge\lim_{n\to\infty}Y_{k_n:n}((X_1,X_2,\ldots))|\mathcal{I}^{\mathbb{X}}\Big)\\
&=\mathbb{E}_{\mathbb{P}}\Big(I(Y\ge\lim_{n\to\infty}Y_{k_n:n})((X_1,X_2,\ldots))|\mathcal{I}^{\mathbb{X}}\Big)\\
&=\mathbb{E}_{\mathbb{Q}}\Big(I(Y\ge\lim_{n\to\infty}Y_{k_n:n})|\tilde{\mathcal{I}}\Big)((X_1,X_2,\ldots))\quad\mathbb{P}-\textrm{a.s.},
\end{align*}
where the first and the third equalities follow from the definition of~$Y$ and Lemma~\ref{fakt2}, respectively.
Now it suffices to observe that
\[
\mathbb{E}_{\mathbb{Q}}\Big(I(Y\ge\lim_{n\to\infty}Y_{k_n:n})|\tilde{\mathcal{I}}\Big)((X_1,X_2,\ldots))\ge 1-\lambda\quad\mathbb{P}-\textrm{a.s.},
\]
because by~\eqref{def_q} and~\eqref{pomoc}, we have
\begin{align*}
\mathbb{P}\Big(\mathbb{E}_{\mathbb{Q}}\Big(I(Y&\ge\lim_{n\to\infty}Y_{k_n:n})|\tilde{\mathcal{I}}\Big)((X_1,X_2,\ldots))\ge 1-\lambda\Big)\\
&=\mathbb{Q}\Big(\mathbb{E}_{\mathbb{Q}}(I(Y\ge\lim_{n\to\infty}Y_{k_n:n})|\tilde{\mathcal{I}})\ge 1-\lambda\Big) =1.
\end{align*}
Condition~3 can be proved in an~analogous way as Condition~2.
\end{proof}

\begin{proof}[Proof of Theorem~\ref{tw.2}]
Let the measure $\mathbb{Q}$ and the rv~$Y$ be as in Theorem~\ref{tw.6}. Then Theorems~\ref{dodatkowy1} and~\ref{fakt3} show that the condition that the conditional quantile $\pi_{\lambda}(X|\mathcal{I}^{\mathbb{X}})$ is unique implies that so is $\pi_{\lambda}(Y|\mathcal{I}_{\mathbb{Q}})$. Hence assumptions of Lemma~\ref{lem4.2} are satisfied and using this result we obtain the desired conclusion.
\end{proof}
%Now, using the above-mentioned facts, we can establish the more elegant version of Lemma~\ref{lem4.2}.


\begin{thebibliography}{99}

\bibitem{bah1966} Bahadur RR (1966)
A note on quantiles in large samples.
Ann Math Statist 37: 577-580

\bibitem{leslie} Buchinsky M, Leslie P (2010)
Educational attainment and the changing U.S. wage structure: Dynamic
implications without rational expectations.
J Labor Econ 28(3): 541-594

\bibitem{bur2018} Buraczy\'nska A, Dembi\'nska A (2018)
A strong ergodic theorem for extreme and intermediate order statistics. 
J Math Anal Appl 460(1): 382-399

\bibitem{cast2013} Castellano K, Ho A (2013)
Contrasting OLS and quantile regression approaches to student growth percentiles. J Educ Behav Stat 38(2): 190-215

\bibitem{chan1992} Chan L, Lakonishok J (1992)
Robust measurement of beta risk. J Financ Quant Anal 27(2): 265-282

\bibitem{dem2014} Dembi\'nska A (2014)
Asymptotic behaviour of central order statistics from stationary processes.
Stoch Process Appl 124: 348-372

\bibitem{dem2017} Dembi\'nska A (2017)
An ergodic theorem for proportions of observations that fall into random sets determined by sample quantiles. 
Metrika 80: 319-332

\bibitem{dembur} Dembi\'nska A, Buraczy\'nska A (2019)
The long-term behavior of number of near-maximum insurance claims.
ArXiv:1904.03169

\bibitem{engle2004} Engle R, Manganelli S (2004)
Conditional value at risk by quantile regression. 
J Bus Econ Stat 22(4): 367-381

\bibitem{ghosh} Ghosh JK (1971)
A new proof of the Bahadur representation of quantiles and an application. 
Ann Math Statist 42: 1957-1961

\bibitem{ghosh2006} Ghosh YN, Mukherjee B (2006)
On probabilistic properties of conditional medians and quantiles. 
Statist. Probab. Lett. 76: 1775-1780

\bibitem{glosten1993} Glosten LR, Jagannathan R, Runkle DE (1993)
On the relation between the expected value and the volatility of the nominal excess return on stocks. 
J. Finance 48(5): 1779-1801

\bibitem{kiefer} Kiefer J (1967)
On Bahadur’s representation of sample quantiles. 
Ann Math Statist 38: 1323-1342

\bibitem{koe2005} Koenker R (2005)
Quantile Regression, Econometric Society Monographs, Cambridge University Press, New York

\bibitem{koe1978} Koenker R, Bassett G (1978)
Regression quantiles.
Econometrica 46(1): 33-50

\bibitem{li2007} Li Q, Racine JS (2007)
Nonparametric econometrics: Theory and practice, Princeton University Press

\bibitem{shir1996} Shiryaev AN (1996)
Probability, 2nd ed. Springer, New York

\bibitem{smirnov} Smirnov NV (1952)
Limit distributions for the terms of a variational series.
Amer Math Soc Transl Ser 1 11: 82-143.
Original published in 1949

%N.V. Smirnov, Limit distributions for the terms of a variational series, Amer. Math. Soc. Transl. Ser. 1 11 (1952) 82–143. Original published in 1949.

\bibitem{spokoiny2013} Spokoiny V, Wang W, H\"ardle WK (2013)
Local quantile regression.
J. Stat. Plan. Inference 143(7): 1109-1129

\bibitem{tomkins75} Tomkins RJ (1975)
On conditional medians.
Ann Probab 3: 375-379

\bibitem{tomkins78} Tomkins RJ (1978)
Convergence properties of conditional medians.
Canad J Statist 6: 169-177

\bibitem{wu2005} Wu WB (2005)
On the Bahadur representation of sample quantiles for dependent sequences.
Ann Statist 33 (4): 1934-1963

\end{thebibliography}
\end{document}